\documentclass[11pt]{article}
\usepackage{amssymb,bbm,amsthm,mathtools}
\newtheorem*{condition}{Condition}
\newtheorem{theorem}{Theorem}
\newtheorem{lemma}{Lemma}
\begin{document}

\title{Bertrand's Postulate for Carmichael Numbers}
\author{Daniel Larsen}
\date{}
\maketitle

\begin{abstract}
Alford, Granville, and Pomerance proved in \cite{AGP:1994} that there are infinitely many Carmichael numbers. In the same paper, they ask if a statement analogous to Bertrand's postulate could be proven for Carmichael numbers. In this paper, we answer this question, proving the stronger statement that for all $\delta>0$ and $x$ sufficiently large in terms of $\delta$, there exist at least $e^{\frac{\log x}{(\log\log x)^{2+\delta}}}$ Carmichael numbers between $x$ and $x+\frac{x}{(\log x)^{\frac{1}{2+\delta}}}$.
\end{abstract}

\section{Introduction}
A composite number $n$ which divides $a^n-a$ for every integer $a$ is called a Carmichael number (named after Robert Carmichael, although the concept is decades older \cite{S:1885}). In a celebrated paper, Alford, Granville, and Pomerance proved that there are infinitely many Carmichael numbers \cite{AGP:1994}. In fact, they were able to prove a much stronger result, namely that there exist at least $x^{2/7}$ of them up to sufficiently large $x$. Despite such a strong lower bound on the number of Carmichael numbers up to $x$ (which has subsequently been improved to above $x^{1/3}$ \cite{H:2008}), little is currently known about their distribution. Indeed, in their paper, Alford, Granville, and Pomerance note, ``It seems to be difficult to prove a `Bertrand's postulate for Carmichael numbers,' that is, that there is always a Carmichael number between $x$ and $2x$ once $x$ is sufficiently large.'' In this paper, we settle this question, in fact proving that the ratio of successive Carmichael numbers tends to 1.

The method of Alford, Granville, and Pomerance is to choose $L$ such that $(\mathbb{Z}/L\mathbb{Z})^{\times}$ has no elements of large order and then find $k$ such that there are many primes of the form $dk+1$ where $d$ divides $L$. For the rest of the argument to work, it is important that the number of divisors which produce primes is much bigger than the order of any element in the multiplicative group mod $L$. The maximal order of elements of $(\mathbb{Z}/L\mathbb{Z})^{\times}$, assuming $L$ is square-free, is the least common multiple of the numbers of the form $p-1$ where $p$ is a prime divisor of $L$. Therefore, $L$ is (roughly speaking) chosen to be the product of primes less than $y$ such that $p-1$ possesses only small prime factors.

The next step is to use character theory to show that there are subsets of the primes of the form $dk+1$ such that their product is 1 mod $L$. Call one such product $\Pi$. Assuming $k$ is relatively prime to $L$, $\Pi$ is in fact 1 mod $kL$ (since each prime factor of $\Pi$ is 1 mod $k$). By Korselt's criterion \cite{K:1899}, a square-free number is a Carmichael number if and only if $p-1\mid n-1$ whenever $p\mid n$. Since $dk\mid kL$, it follows that $\Pi$ is a Carmichael number.

A critical step of their argument is finding a suitable $k$. To do this, they look at all the pairs $(k,d)$ such that $dk+1$ is prime, where $k$ is drawn from some bounded interval and $d$ is chosen from $\bold{D}$, some subset of the divisors of $L$. Assuming $\bold{D}$ contains enough elements, there exists a particular $k$ with many primes of the form $dk+1$ $(d\in \bold{D})$ by the pigeonhole principle. This is a rough sketch of the method used by Alford, Granville, and Pomerance, although of course their actual method is considerably more complicated.

Unfortunately, this approach cannot be used to produce Carmichael numbers in short intervals. This is because the use of the pigeonhole principle relinquishes control over which elements of $\bold{D}$ correspond to primes. In the extreme case, this could result in all the primes being almost exactly the same size, in which case, there might not be any products in a particular short interval.

Here is a flawed approach which nevertheless points in the right direction. Split the divisors of $L$ into pairs of the form $(d_i,d'_i)$ with the property that $d_i>2d'_i$. (Of course, some divisors will need to be discarded to make this work, but that is of no concern.) Then use the pigeonhole principle on triples $(k,d_i,d'_i)$ where both $d_ik+1$ and $d'_ik+1$ are prime. This guarantees that no more than half of the divisors corresponding to primes will be of the same general size. In fact, as long as $k$ is chosen properly, the products of the primes will spread out evenly.

The problem with this method is there are no known lower bounds on the number of $k$ such that $dk+1$ and $d'k+1$ are both prime for fixed $d$ and $d'$. However, there has been considerable progress in this direction. In a landmark 2014 paper \cite{Z:2014}, Zhang showed that there exists $h\leq 70000000$ such that $x$ and $x+h$ are simultaneously prime infinitely often. Then in a further breakthrough, Maynard \cite{M:2015} and independently Tao (unpublished) found a different and more flexible approach to the problem, which allowed more general results to be proven.

The connection between prime tuples and Carmichael numbers has long been known. In 1939, Chernick observed that $(6k+1)(12k+1)(18k+1)$ is a Carmichael number whenever those three factors are all prime \cite{C:1939}. The prime tuples conjecture (or rather a slightly more general statement known as Dickson's conjecture) asserts that every finite set of linear forms can be made simultaneously prime infinitely often unless the set is inadmissible, which means that some prime divides all the values of their product at integers. Clearly, then, the prime tuples conjecture would imply that there are infinitely many Carmichael numbers. In fact, it would imply something stronger (and as yet unproven), that there are infinitely many Carmichael numbers with a bounded number of prime factors. This result can actually be proven under the assumption of hypotheses which are much weaker than the prime tuples conjecture, although still unproven. (See \cite{W:2019}.) However, to the author's knowledge, none of the unconditional progress towards the prime tuples conjecture had ever been used to answer questions regarding Carmichael numbers before this paper.

We are now in a position to revise our previous strategy. We split the divisors into $M$-tuples for some large number $M$, instead of just into pairs. A 2016 paper of Maynard \cite{M:2016}, which extends his previous work, will then give us lower bounds on the number of triples $(k,d,d')$ with $d$ and $d'$ in the same $M$-tuple and $dk+1$ and $d'k+1$ both prime, producing enough triples that the pigeonhole principle will be able to find a specific $k$ with many primes of the form $dk+1$. Crucially, this can be done while making sure the primes do not cluster around powers of any $\alpha>1+\epsilon$. (Here $\epsilon$ represents some positive quantity which goes to zero as $y$ (the starting parameter) goes to infinity.) Then a standard argument from Fourier analysis will guarantee Carmichael numbers of controllable size. By slowly sliding $y$ to infinity, the following theorem can be proven.
\begin{theorem}
\label{main}
For all $\delta>0$ and $x$ sufficiently large in terms of $\delta$, there are at least $e^{\frac{\log x}{(\log\log x)^{2+\delta}}}$ Carmichael numbers between $x$ and $x+\frac{x}{(\log x)^{\frac{1}{2+\delta}}}$.
\end{theorem}
\section{Preliminaries}
To start with, we give some standard notation in analytic number theory. Let $\mathbb{P}$ be the set of primes. Let $(a,b)$ be the greatest common divisor of $a$ and $b$, and let $\phi(n)=\#\{0<m<n: (m,n)=1\}$. Let \[\pi(x)=\sum_{p\leq x}1,\] \[\pi(x; q, a)=\sum_{\substack{p\leq x\\ p\equiv a\pmod{q}}}1,\] and \[E(x; q, a)=\Big\vert \pi(x; q, a)-\frac{\pi(x)}{\phi(q)}\Big\vert.\] In the above expressions, $p$ refers to a prime. Indeed, throughout the paper, $p$ will always represent a prime. In contrast, $q$ in the above expressions is not necessarily prime.

The final arithmetic function we will use is Carmichael's $\lambda$ function, defined on $n$ as the largest order of an element that appears in $(\mathbb{Z}/n\mathbb{Z})^{\times}$. We start with a standard result about $\phi$.
\begin{lemma}
\label{Phi}
\[\sum_{k\leq x}\frac{1}{\phi(k)}=O(\log x).\]
\vskip -30pt
\end{lemma}
\begin{proof}
This is proved in, for example, Lemma 12.1 of \cite{LP:2019}, where it is further shown that the constant of proportionality is $\frac{\zeta(2)\zeta(3)}{\zeta(6)}$ where $\zeta$ is the Riemann zeta function.
\end{proof}
Maynard's sieve requires a version of the Bombieri-Vinogradov theorem. The standard error term of $\frac{x}{(\log x)^A}$ won't suffice for the purposes of this paper. However, a better bound can be attained if the effects of at most one Siegel zero are disregarded.
\begin{theorem}
\label{BV}
There exists a positive constant $C$ such that for every sufficiently large value of $x$, there exists an integer $s\in [\sqrt{\log x},e^{\sqrt{\log x}}]$ with \[\sum_{\substack{q\leq x^{2/5}\\ s\nmid q}}\max_{2\leq t\leq x}\max_{(a,q)=1}\Big\vert\pi(t; q,a)-\frac{\pi(t)}{\phi(q)}\Big\vert\leq \frac{x}{e^{C\sqrt{\log x}}}.\]
\vskip -30pt
\end{theorem}
\begin{proof}
This follows from Theorem 11.2 from \cite{LP:2019} by summation by parts.
\end{proof}
A result very much like Theorem \ref{BV} also appears in \cite{LP:2019}, but the above version is slightly more general. For the rest of the paper, we fix a particular value of $s$ for each $x$, which we label $s_x$. Now, given a set of linear forms, consider its product, which will be some polynomial $P(x)$. If there is no prime $p$ such that $p\mid P(x)$ for all integer values of $x$ then we say that the set of linear forms is \emph{admissible}.
\begin{theorem}
\label{Maynard}
For every $m\in \mathbb{N}$ and $\epsilon>0$, there exists a positive integer $N$ such that for every admissible set \[\bold{L}=(L_1(k),L_2(k),\ldots,L_N(k))=(d_1k+c_1,d_2k+c_2,\ldots,d_Nk+c_N)\] of $N$ linear forms where $d_i>c_i>0$ for all $i$, \begin{equation}\label{theorem}\#\{x\leq k<2x: \#\{i: L_i(k)\in \mathbb{P}\}\geq m\}\geq \frac{x}{(\log x)^N}\end{equation} for sufficiently large $x$ in terms of $\bold{L}$. In fact, \eqref{theorem} holds for every $x$ with an associated $d_+$ where the pair of $x$ and $d_+$ satisfy the following properties:
\begin{enumerate}
\item $x$ is sufficiently large in terms of $m$ and $\epsilon$.
\item $d_+\geq d_i$ for all $i$.
\item $\log x\geq (\log d_+)^{2+\epsilon}$
\item $s_{2xd_+}$ (as taken from the statement of Theorem \ref{BV}) has a prime factor which does not divide $d_i$ for any $i$.
\end{enumerate}
\end{theorem}
\begin{proof}
This theorem is a particular case of Theorem 3.1 from \cite{M:2016}, corresponding to $\bold{P}=\mathbb{P}$ and $\bold{A}=\mathbb{N}$, as well as $\alpha=1$, $\theta=\frac{1}{3}$, $\delta=\frac{1}{6}$, and $B$ equal to the prime factor of $s_{2xd_+}$ which does not divide any of the $d_i$. Some work is still required, though, as various conditions must be checked. One beautiful aspect of Maynard's result is that it turns regularity conditions regarding the individual linear forms into a statement about how the linear forms interact with each other. The two conditions which need to be established for each $i$ are, letting \[\Pi(x; q, a)=\#\{d_ik+c_i\in\mathbb{P}: x\leq k<2x, \text{ }k\equiv a\pmod{q}\}\] and \[\Pi(x)=\#\{d_ik+c_i\in\mathbb{P}: x\leq k<2x\},\] that \begin{equation}\label{condition1}\sum_{\substack{q\leq x^{\theta}\\ (q,B)=1}}\max_{(d_ia+c_i,q)=1}\Big\vert \Pi(x; q, a)-\Pi(x)\frac{\phi(d_i)}{\phi(d_iq)}\Big\vert\ll \frac{\Pi(x)}{(\log x)^{100N^2}}\end{equation} and \begin{equation}\label{condition2}\frac{\phi(B)}{B}\frac{\phi(d_i)}{d_i}\Pi(x)\geq\frac{\delta x}{\log x}.\end{equation} In fact, the latter statement need only be true on average, but we will show that it is true for all $i$. The other conditions for Maynard's theorem are easily seen to be true; for example, the other two conditions of Hypothesis 1 require $\bold{A}$ to be well-distributed in arithmetic progressions, which is definitely true for $\bold{A}=\mathbb{N}$.

Fix $i$ for which we will prove equations \eqref{condition1} and \eqref{condition2}. Let \[E=\sum_{\substack{q\leq x^{1/3}\\ (q,B)=1}}\max_{(d_ia+c_i,q)=1}\Big\vert\Pi(x; q, a)-\frac{\phi(d_i)}{\phi(d_iq)}\Pi(x)\Big\vert.\] To clean up the following equations, we will drop the subscripts from $d_i$ and $c_i$. First of all, we claim that \[\Pi(x; q, a)=\pi(2dx; dq,da+c)-\pi(dx; dq, da+c).\] Indeed, $\Pi(x; q, a)$ counts the number of primes of the form $d(a+nq)+c$ with $x\leq a+nq<2x$, or equivalently, $dx+c\leq d(a+nq)+c<2dx+c$. Since $0<c<d$ by hypothesis, this is equivalent to $dx<d(a+nq)+c<2dx$, which establishes our claim. Now \begin{align*}E&=\sum_{\substack{q\leq x^{1/3}\\ (q,B)=1}}\max_{(da+c,q)=1}\Big\vert \pi(2dx; dq, da+c)-\pi(dx; dq, da+c)-\frac{\phi(d)}{\phi(dq)}\big(\pi(2dx; d,c)-\pi(dx; d, c)\big)\Big\vert\\&\leq \sum_{\substack{q\leq x^{1/3}\\ (q,B)=1}}\max_{(b,dq)=1}\Big\vert \pi(2dx; dq, b)-\pi(dx; dq, b)-\frac{\phi(d)}{\phi(dq)}\big(\pi(2dx; d,c)-\pi(dx; d, c)\big)\Big\vert\end{align*} since $c$ and $d$ are relatively prime (which follows from the admissibility of $\bold{L}$) so that for $b=da+c$ with $(da+c,q)=1$, $b$ is relatively prime to both $q$ and $d$. Now, \begin{align*}\Big\vert\pi(dx; dq, b)-\frac{\phi(d)}{\phi(dq)}\pi(dx; d, c)\Big\vert&\leq \Big\vert \pi(dx; dq, b)-\frac{\pi(dx)}{\phi(dq)}\Big\vert+\frac{\phi(d)}{\phi(dq)}\Big\vert\pi(dx; d, c)-\frac{\pi(dx)}{\phi(d)}\Big\vert\\&=E(dx; dq, b)+\frac{\phi(d)}{\phi(dq)}E(dx; d, c),\end{align*} and the analogous inequality with $2dx$ instead of $dx$ holds as well. Hence \[E\leq  \sum_{\substack{q\leq x^{1/3}\\ (q,B)=1}}\max_{(b,dq)=1}\big(E(2dx; dq, b)+E(dx; dq, b)\big)+\big(E(2dx; d, c)+E(dx; d, c)\big)\sum_{\substack{q\leq x^{1/3}\\ (q,B)=1}}\frac{\phi(d)}{\phi(dq)}.\] By definition, $B$ is relatively prime to $d$, so the first part of the above expression is bounded by \[\sum_{\substack{q\leq dx^{1/3}\\ (q,B)=1}}\max_{(b,q)=1}\big(E(2dx; q, b)+E(dx; q, b)\big),\] while the second part is less than \[\big(E(2dx; d,c)+E(dx; d, c)\big)\sum_{q\leq x^{1/3}}\frac{1}{\phi(q)}\] since $\phi(d)\phi(q)\leq \phi(dq)$. For $x\leq z\leq 2x$, Theorem \ref{BV} implies that \begin{align*}\sum_{\substack{q\leq dx^{1/3}\\ (q,B)=1}}\max_{(a,q)=1}E(dz; q, a)&\leq \sum_{\substack{q\leq (2d_+x)^{2/5}\\ (q,B)=1}}\max_{2\leq t\leq 2d_+x}\max_{(a,q)=1}E(t; q, a)\leq \frac{2d_+x}{e^{C\sqrt{\log(2d_+x)}}}\\&\leq \frac{2x}{e^{C\sqrt{\log x}-(\log x)^{\frac{1}{2+\epsilon}}}}\ll \frac{x}{(\log x)^{100N^2+1}}.\end{align*} Additionally, \[E(dz; d, c)<\sum_{\substack{q\leq dx^{1/3}\\ (q,B)=1}}\max_{(a,q)=1}E(dz; q, a)\ll \frac{x}{(\log x)^{100N^2+1}}\] for $x\leq z\leq 2x$. Putting everything together (and using Lemma \ref{Phi}), \begin{align*}E&\leq\sum_{\substack{q\leq dx^{1/3}\\ (q,B)=1}}\max_{(b,q)=1}\big(E(2dx; q, b)+E(dx; q, b)\big)+\big(E(2dx; d,c)+E(dx; d, c)\big)\sum_{q\leq x^{1/3}}\frac{1}{\phi(q)}\\&\ll\frac{x}{(\log x)^{100N^2+1}}+\frac{x}{(\log x)^{100N^2+1}}+\log x\Big(\frac{x}{(\log x)^{100N^2+1}}+\frac{x}{(\log x)^{100N^2+1}}\Big)\\&\ll \frac{x}{(\log x)^{100N^2}}.\end{align*} This shows that Hypothesis 1 is satisfied. Additionally, \begin{align*}\Pi(x)&=\pi(2dx; d, c)-\pi(dx; d, c)\geq \frac{\pi(2dx)-\pi(dx)}{\phi(d)}-E(2dx; d, c)-E(dx; d, c)\\&>\frac{dx}{2\phi(d)\log x}-O\Big(\frac{x}{(\log x)^{100N^2+1}}\Big)>\frac{dx}{3\phi(d)\log x}\end{align*} for sufficiently large $x$, using the Prime Number Theorem. Since $B$ is a prime, $\phi(B)\geq \frac{B}{2}$ so \[\frac{\phi(B)}{B}\frac{\phi(d)}{d}\Pi(x)>\frac{x}{6\log x}.\] We have thus established equations \eqref{condition1} and \eqref{condition2}. The theorem now follows from Theorem 3.1 of \cite{M:2016} by taking $N$ to be large compared to the $C$ corresponding to $\alpha=1$ and $\theta=\frac{1}{3}$ in Maynard's result.
\end{proof}
\section{Choosing Primes}
For a given set $\bold{S}$, let $\bold{D}_{\bold{S}}=\{d: d\mid \prod_{s\in\bold{S}}s\}$ and $n\bold{S}=\{ns: s\in \bold{S}\}$. Let the \emph{logarithmic spread} of $S$ be the smallest non-zero value $\vert \log s_1-\log s_2\vert$ takes as $s_1$ and $s_2$ range over $\bold{S}$, and let the \emph{logarithmic diameter} be the largest value the same expression takes.

Now, suppose $E>0$ and $\gamma>0$ are such that \[\Big\{p\leq x: P^+(p-1)\leq x^{1-E}\Big\}\geq \frac{\gamma x}{\log x}\] for sufficiently large $x$, where $P^+(n)$ is the largest prime factor of $n$. Indeed, such values exist and can be computed effectively, as in \cite{AGP:1994}. (This, in fact, goes back to work of Erd\H{o}s \cite{E:1935}.) Unlike in the paper of Alford, Granville, and Pomerance, the exact values of $E$ and $\gamma$ will make no difference in the end, so we will not specify them. Note that the smaller $\delta$ is taken in Theorem \ref{main}, the stronger the result becomes. Therefore, for the rest of the paper, we will assume that $\delta$ is some small but positive quantity. Let $y$ be a large parameter (large in terms of both $E$ and $\delta$). Let $\Delta=\frac{\delta}{24}$. Let $M$ be the smallest power of two greater than the value of $N$ in Theorem \ref{Maynard} corresponding to $m=2$ and $\epsilon=\Delta$. Let $\Upsilon=\big\lceil e^{y^{2+2\Delta}}\big\rceil$ and \[L^{\ast}=\prod_{\substack{\frac{y}{\log y}\leq p\leq y\\ P^+(p-1)\leq y^{1-E}}}p.\] Let $p^{\ast}$ be the largest prime factor of $s_x$ (as defined following Theorem \ref{BV}) for $x=2\Upsilon (L^{\ast})^2$. Finally, let \[\bold{P}=\Big\{\frac{y}{\log y}\leq p\leq y: P^+(p-1)\leq y^{1-E},\text{ }p\neq p^{\ast}\Big\}.\] Taking $\eta=\frac{\gamma}{2}$, it is evident that $\vert \bold{P}\vert\geq \frac{\eta y}{\log y}$. We now give a combinatorial lemma about $\bold{D}_{\bold{P}}$ which explains why we took $M$ to be a power of two.
\begin{lemma}
There exists a subset of $\bold{D}_{\bold{P}}$ (containing half as many elements) which can be partitioned into $\frac{2^{\vert \bold{P}\vert}}{2M}$ sets of $M$ elements, each with logarithmic spread greater than 1.
\end{lemma}
\begin{proof}
Let $k$ be the smallest integer such that $2^{2k}>2M\binom{2k}{k}$. Since $M$ is fixed, we can safely assume that $y$ is arbitrarily large in terms of $k$. Let $\bold{P}_0$ be an arbitrary subset of $\bold{P}$ with $2k$ elements. Then any subset of $\bold{D}_{\bold{P}_0}$ with at least $\frac{2^{\vert\bold{P}_0\vert}}{2M}$ elements contains two elements such that one of them has strictly more prime factors than the other. When $y$ is large, such a subset of $\bold{D}_{\bold{P}_0}$ will have logarithmic diameter greater than 1, considering that the primes in $\bold{P}_0$ are within a factor of $\log y$ of each other. We will say that a set $\bold{S}$ possesses ``the wide subset property'' if every subset of $\bold{D}_{\bold{S}}$ with at least $\frac{2^{\vert\bold{S}\vert}}{2M}$ elements has logarithmic diameter greater than 1. In particular, $\bold{P}_0$ has this property.

Now let $\bold{P}_1=\bold{P}_0\cup p_1$ for some $p_1\in \bold{P}\setminus \bold{P}_0$. Suppose $\bold{R}$ is a subset of $\bold{D}_{\bold{P}_1}$ with at least $\frac{2^{\vert\bold{P}_1\vert}}{2M}$ elements. Decompose $\bold{R}$ as $\bold{R}_1\cup (p\bold{R}_2)$ where $\bold{R}_1=\bold{R}\cap \bold{D}_{\bold{P}_0}$ and $\bold{R}_2=(p^{-1}\bold{R})\cap \bold{D}_{\bold{P}_0}$. Either $\bold{R}_1$ or $\bold{R}_2$ (perhaps both) has at least $\frac{2^{\vert P_0\vert}}{2M}$ elements, and they are subsets of $\bold{D}_{\bold{P}_0}$. Hence, at least one of $\bold{R}_1$ and $\bold{R}_2$ has logarithmic diameter greater than 1, which means that $\bold{R}$ has logarithmic diameter greater than 1. Hence $\bold{D}_{\bold{P}_1}$ also has ``the wide subset property.'' By induction, $\bold{D}_{\bold{P}}$ has this property as well.

Recall that $M$ is a power of two. It therefore makes sense to define $\bold{D}_1$ as the set consisting of the $\frac{\vert \bold{D}_{\bold{P}}\vert}{2M}$ smallest elements of $\bold{D}_{\bold{P}}$, $\bold{D}_2$ as the set consisting of the next $\frac{\vert \bold{D}_{\bold{P}}\vert}{2M}$ smallest elements, and so on, resulting in sets $\bold{D}_1, \bold{D}_2, \ldots, \bold{D}_{2M}$. Let $\bold{S}_j$ be the set which contains the $j$th element of $\bold{D}_2, \bold{D}_4, \ldots, \bold{D}_{2M}$. Since $\bold{D}_{\bold{P}}$ has ``the wide subset property,'' $\bold{D}_3, \bold{D}_5, \ldots, \bold{D}_{2M-1}$ all have logarithmic diameter greater than 1. Therefore $\bold{S}_j$ has logarithmic spread greater than 1 for all $j$. As promised, we now have $\frac{\vert \bold{D}_{\bold{P}}\vert}{2M}$ sets, each containing $M$ different elements of $\bold{D}_{\bold{P}}$ and each with logarithmic spread greater than 1.
\end{proof}
Let $L=\prod_{p\in\bold{P}}p$ and $Y=\Upsilon L$. Fix $j$. Let $\bold{T}_j$ be the set of integers $k$ between $Y$ and $2Y$ with $(k,L)=1$ such that there exist $d$ and $d'$, distinct elements of $\bold{S}_j$, with both $dk+1$ and $d'k+1$ prime. Label the elements of $\bold{S}_j$ as $d_1, d_2, \ldots, d_M$. Let \[\bold{\Omega}_p=\{a\in(\mathbb{Z}/p\mathbb{Z})^{\times}: (d_ia+1,p)=1\text{ }\forall i\}\] for $p\mid L$ and then let \[\bold{\Omega}=\prod_{p\mid L}\bold{\Omega}_p.\] By the Chinese remainder theorem, each element of $\bold{\Omega}$ can be thought of as an integer $a_L$ with $1\leq a_L<L$ such that for every $p\mid L$, $a_L\equiv a_p\pmod{p}$, where $a_p$ is the element in $\bold{\Omega}_p$ which corresponds to $a_L$.
\begin{lemma}
For all $a_L\in\bold{\Omega}$, the set of linear forms $\{(d_iL)k+(d_ia_L+1)\}_{i=1}^M$ is admissible.
\end{lemma}
\begin{proof}
If $p\mid L$ then $p\nmid d_ia_L+1$ for all $i$, so $p$ never divides $(d_iL)k+(d_ia_L+1)$ regardless of $i$ and $k$. On the other hand, if $p\nmid L$ then $L$ has a multiplicative inverse $L'$ in $(\mathbb{Z}/p\mathbb{Z})^{\times}$, and when $k\equiv -L'a_L\pmod{p}$, \[(d_iL)k+(d_ia_L+1)\equiv d_i(-LL'a_L+a_L)+1\equiv d_i(-a_L+a_L)+1\equiv 1\pmod{p}\] for all $i$. Hence, no prime always divides one of the elements of the given set of linear forms.
\end{proof}
The reason for using this set of linear forms as opposed to $\{d_ik+1\}_{i=1}^M$ is we want to ensure that $k$ is relatively prime to $L$. However, we also can't afford to only restrict our attention to $\{(d_iL)k+(d_ia_L+1)\}_{i=1}^M$ for a particular value of $a_L$. Therefore, we will use all values of $a_L\in \bold{\Omega}$, almost breaking even as a result since \begin{align*}1-\frac{\vert \bold{\Omega}\vert}{L}&\leq 1-\prod_{p\mid L}\Big(1-\frac{M+1}{p}\Big)=1-e^{-O(1)\sum_{p\mid L}\frac{1}{p}}\\&=1-e^{-O(1)(\log\log y-\log\log\frac{y}{\log y})}=1-e^{-o(1)}=o(1).\end{align*}

Fix a particular $a_L\in \bold{\Omega}$. We apply Theorem \ref{Maynard}, taking $m=2$ and $\epsilon=\Delta$. Recall that $M$ is at least as large as the value of $N$ in the theorem corresponding to these values of $m$ and $\epsilon$. Therefore, we can use the linear forms $\{(d_iL)k+(d_ia_L+1)\}$ for $1\leq i\leq M$. We will also take $x=\Upsilon$ and $d_+=(L^{\ast})^2$. Indeed, with these choices of variables, \[e^{(\log d_+)^{2+\epsilon}}\leq e^{(\log(y^2L^2))^{2+\Delta}}<e^{(\log(y^{\frac{4y}{\log y}}))^{2+\Delta}}=e^{(4y)^{2+\Delta}}<e^{y^{2+2\Delta}}\leq x\] so the conditions of the theorem do indeed hold. Now, this implies that $\bold{T}_j$ has at least $\frac{\Upsilon}{(\log \Upsilon)^M}$ elements of the form $Lk+a_L$. We can then repeat this process for every element in $\bold{\Omega}$, getting different elements of $\bold{T}_j$ each time. Therefore, \[\vert \bold{T}_j\vert\geq \frac{\Upsilon\vert\bold{\Omega}\vert}{(\log \Upsilon)^M}=(1-o(1))\frac{\Upsilon L}{(\log \Upsilon)^M}\geq (1-o(1))\frac{Y}{(\log Y)^M}.\]

We now have a large set of primes, divided into pairs $dk+1$ and $d'k+1$. Before applying the Pigeonhole Principle to find a particular $k$ with many such pairs of primes, we need to filter out those values of $d$, $d'$, and $k$ for which $dk+1$ and $d'k+1$ ``cluster.'' To do this, we introduce two parameters, $V$ and $W$. Recall that the goal is to ensure that there is no $\alpha$ slightly greater than 1 such that the primes are very close to powers of $\alpha$. Roughly speaking, $V$ defines what ``very close'' means, while $W$ defines what ``slightly greater than 1'' means.

To make things more precise, consider, for some real numbers $V$ and $W$ which will chosen later, the set $\bold{U}_j$ of $k$ between $Y$ and $2Y$ such that there exist $d$ and $d'$, two elements of $\bold{S}_j$, with $d>d'$, as well as $\omega$ with $\vert\omega\vert<W$, such that \[\big\vert \omega\log(dk+1)-m\big\vert\leq\frac{1}{V}\] and \[\big\vert\omega\log(d'k+1)-m'\big\vert\leq\frac{1}{V}\] for some non-zero integers $m$ and $m'$. Heuristically, for any given choice of $\omega$, the chance of a random integer $k$ satisfying these two equations is $\frac{1}{V^2}$. On the other hand, this chance essentially resets every time $\omega$ goes up by $\frac{1}{V\log Y}$. Hence the chance that for a random integer $k$, there exists $\omega$ satisfying these two equations is roughly $\frac{W \log Y}{V}$. If $W$ is significantly smaller than $V$, one would expect that $\bold{U}_j$ would have relatively few elements. We now work to make this heuristic rigorous, starting by formalizing the two relevant properties that the sequences $\log(dk+1)$ and $\log(d'k+1)$ possess.
\begin{condition}[Parallel]
Two sequences $\{a_i\}_{i=1}^n$ and $\{a'_i\}_{i=1}^n$ are said to satisfy the Parallel Condition if there exists $\alpha>1$ such that $\vert a_i-a'_i-\alpha\vert\leq \frac{1}{n}$ for all $i$.
\end{condition}
\begin{condition}[Smoothness]
A sequence $\{a_i\}_{i=1}^n$ is said to satisfy the Smoothness Condition if the number of $i$ such that $\vert a_i-\beta\vert\leq \frac{1}{n}$ is at most 12 for all $\beta$.
\end{condition}
\begin{lemma}
\label{clustering}
Suppose $A>1$ is some real number and $\{a_i\}_{i=1}^n$ and $\{a'_i\}_{i=1}^n$ are sequences of real numbers between 1 and $A$ which together satisfy the Parallel Condition. Assume further that $\{a_i\}$ satisfies the Smoothness Condition. For any positive real numbers $W$ and $V$ with $VW\ll n$, the number of $i$ for which there exist $\omega$ with $\vert \omega\vert<W$ and non-zero integers $m$ and $m'$ such that \[\vert \omega a_i-m\vert\leq \frac{1}{V}\] and \[\vert \omega a'_i-m'\vert\leq \frac{1}{V}\] is $O(\frac{A^3W\log(AW)n}{V})$.
\end{lemma}
\begin{proof}
For integers $\ell$ and $m$, let $\bold{K}_{\ell,m}$ be the number of indices $i$ for which there exists $\omega$ with $\vert \omega\vert<W$ such that \[\vert \omega a_i-m\vert\leq \frac{1}{V}\] and \[\vert \omega a'_i-m+\ell\vert\leq \frac{1}{V},\] where both $m$ and $m-\ell$ are both non-zero. Fix integers $\ell$ and $m$ with $m$ and $m-\ell$ both non-zero. Suppose $i_1$ and $i_2$ are two indices with corresponding $\omega_1$ and $\omega_2$ less than $W$ in absolute value such that \[\vert \omega_k a_{i_k}-m\vert\leq \frac{1}{V}\] and \[\vert \omega_ka'_{i_k}-m+\ell\vert\leq \frac{1}{V}\] for both choices of $k\in \{1,2\}$. Then \[\vert \omega_k(a_{i_k}-a'_{i_k})-\ell\vert\leq \frac{2}{V}\] for both $k$. By assumption, there exists $\alpha>1$ such that $\vert a_i-a'_i-\alpha\vert\leq \frac{1}{n}$ for all $i$. It follows that \[\vert\omega_k\alpha-\ell\vert\leq \vert \omega_k(a_{i_k}-a'_{i_k})-\ell\vert+\vert\omega_k(a_{i_k}-a'_{i_k}-\alpha)\vert\leq \frac{2}{V}+\frac{W}{n}=O\big(\frac{1}{V}\big)\] for both $k$. In particular, \[\vert\omega_1-\omega_2\vert\leq \frac{1}{\alpha}\big\vert (\omega_1\alpha-\ell)-(\omega_2\alpha-\ell)\big\vert<\vert \omega_1\alpha-\ell\vert+\vert \omega_2\alpha-\ell\vert=O\big(\frac{1}{V}\big).\] Now, \begin{align*}\vert a_{i_1}-a_{i_2}\vert&=\frac{\vert (\omega_1a_{i_1}-m)-(\omega_2a_{i_2}-m)+(\omega_2-\omega_1)a_{i_2}\vert}{\vert \omega_1\vert}=O\big(\frac{a_{i_2}}{V\vert \omega_1\vert}\big)\\&=O\big(\frac{a_{i_1}a_{i_2}}{V\vert m\vert}\big)=O\big(\frac{A^2}{V\vert m\vert}\big)\end{align*} since $\vert \omega_1\vert\gg \frac{\vert m\vert}{a_{i_1}}$, using the fact that $m$ is non-zero. By assumption, for all $\epsilon\geq \frac{1}{n}$, the number of $i$ such that $\vert a_i-\beta\vert\leq \epsilon$ is $O(\epsilon n)$ regardless of $\beta$. Let $\epsilon=\frac{A^2}{V\vert m\vert}$. Since $\vert m\vert=O(a_{i_1}\omega_1)=O(AW)$, it follows that \[\epsilon\gg\frac{A}{VW}\gg \frac{A}{n}>\frac{1}{n}.\] Therefore, $\bold{K}_{\ell,m}$ has $O(\frac{A^2n}{V\vert m\vert})$ elements. Hence, setting $\bold{K}=\cup_{\ell,m}\bold{K}_{\ell,m}$, we have that \begin{align*}\vert \bold{K}\vert&=\sum_{\ell=O(AW)}\sum_{\substack{m=O(AW)\\ m\neq 0}}\vert\bold{K}_{\ell,m}\vert=O\big(AW\log(AW)\frac{A^2n}{V}\big)\\&=O\big(\frac{A^3W\log(AW)n}{V}\big)\end{align*} as promised.
\end{proof}
Let $n=Y$ and $A=2\log Y$. Let $V=e^{\frac{\eta y}{(4+2\Delta)\log y}}$ and $W=e^{\frac{\eta y}{(4+4\Delta)\log y}}$. Let $a_k=\log(d(Y+k-1)+1)$ and $a'_k=\log(d'(Y+k-1)+1)$ for $1\leq k\leq n$. The Parallel Condition holds with $\alpha=\log d-\log d'$ since \begin{align*}\log(d(Y+k-1)+1)-\log(d'(Y+k-1)+1)&=\log\Big(\frac{d(Y+k-1)+1}{d'(Y+k-1)+1}\Big)\\&=\log\Big(\frac{d}{d'}+O(\frac{1}{d'Y})\Big)\\&=\alpha+O\big(\frac{1}{nd'}\big).\end{align*} Furthermore, if we define $b_i=a_{i+1}-a_i$ for $1\leq i<n$ then we see that $b_i$ is a decreasing sequence with $b_1<3b_{n-1}$, while $a_i$ is an increasing sequence with $a_n>a_1+\frac{1}{2}$. From these two facts, it follows that the average value of $b_i$ is greater than $\frac{1}{2n}$ and the smallest value of $b_i$ is greater than $\frac{1}{6n}$. This shows that $\{a_i\}$ satisfies the Smoothness Condition. Since our values of $V$ and $W$ satisfy the hypotheses of Lemma \ref{clustering}, we obtain the bound \[\vert\bold{U}_j\vert=\binom{M}{2}O\big(\frac{W\log(W\log Y)\log^3(Y)Y}{V}\big)=o\big(\frac{Y}{(\log Y)^M}\big),\] remembering that $M$ is fixed. Therefore, removing the elements of $\bold{U}_j$ from $\bold{T}_j$ still leaves $(1-o(1))\frac{Y}{(\log Y)^M}$ elements.

Now, so far all this work has been done for a particular choice $j$, but there are $\frac{\vert \bold{D}_{\bold{P}}\vert}{2M}-1$ other values of $j$ which could have been chosen. This gives in total $(1-o(1))\frac{\vert \bold{D}_{\bold{P}}\vert Y}{2M(\log Y)^M}$ pairs $(k,j)$ with the following properties.
\begin{enumerate}
\item $k\notin \bold{U}_j$
\item $(k,L)=1$
\item $Y\leq k<2Y$
\item There exist distinct elements of $\bold{S}_j$, which we will call $d$ and $d'$, such that $dk+1$ and $d'k+1$ are both prime.
\end{enumerate}
Hence, the pigeonhole principle tells us that there exists $k_0$ between $Y$ and $2Y$ relatively prime to $L$ with $(1-o(1))\frac{\vert \bold{D}_{\bold{P}}\vert}{2M(\log Y)^M}$ pairs of ``non-clustering'' primes, each of the form $dk+1$ for $d$ dividing $L$. Let $\bold{Q}$ be the set of these pairs. Then, \[\vert \bold{Q}\vert\geq (1-o(1))\frac{\vert \bold{D}_{\bold{P}}\vert}{2M(\log Y)^M}\geq (1-o(1))\frac{2^{\frac{\eta y}{\log y}}}{2M(\log Y)^M}>e^{\frac{\eta y}{2\log y}}.\] 
\section{Building Carmichael Numbers}
In this section, we will construct Carmichael numbers of controlled size. We start with an application of Fourier analysis.
\begin{theorem}
\label{Fourier}
Let $q_1,q_2,\ldots,q_N$ be a set of primes with product $Q$. Label the largest of the primes $q_0$. Assume that $q_i^2\geq q_0$ for all $i$. Let $\ell$ be some positive integer which is not divisible by any of the $q_i$. Assume that \[N>\max\Big(\log q_0,\lambda(\ell)\log\big(\phi(\ell)\big)\Big)^5.\] Further assume that for all $\omega$ with $\frac{2}{\lambda(\ell)\log q_0}<\vert\omega\vert<4A\log\Phi$ and $\alpha$ with $\vert\alpha\vert\geq \frac{2\pi}{\lambda(\ell)}$, at most half of the primes $q_i$ satisfy \begin{equation}\label{nocluster}\vert \omega\log q_i-k\alpha\vert<8\sqrt{\frac{\log \Phi}{N}}\end{equation} for some integer $k$, where $A$ is some real number greater than 1 and \[\Phi=(AN\phi(\ell)\log q_0)^2.\] For all sufficiently large $N$, the number of values of $d$ which divide $Q$, while satisfying both $d\equiv 1\pmod{\ell}$ and $\vert\log d-\frac{\log Q}{2}-B\vert\leq\frac{1}{2A}$, is $2^{N-O(\sqrt{N}+\log\Phi)}$ for all $B\leq \frac{\sqrt{N}\log q_0}{36}$.
\end{theorem}
\begin{proof}
Let $H=(\mathbb{Z}/\ell\mathbb{Z})^{\times}$. Let $h_i$ be the element of $H$ corresponding to $q_i$. We will construct a sequence of nested groups starting with $H$, which we relabel $G_0$. Starting with $j=1$ and then increasing $j$ by increments of one, let $G_j$ be a proper subgroup of $G_{j-1}$ such that \[\#\{i: h_i\in G_{j-1}\setminus G_j\}<\lambda(G_{j-1})^2\log \big(\#\{i: h_i\in G_{j-1}\}\big\vert G_{j-1}\big\vert\big),\] with the process stopping when no such proper subgroup exists. The result is a chain of proper subgroups \[G_0\supset G_1\supset \cdots \supset G_m\] where for every proper subgroup of $G_m$, at least $\lambda(G_m)^2\log (n\vert G_m\vert)$ of the $h_i$ are in $G_m$ but not in the proper subgroup (letting $n$ be the number of $i$ such that $h_i$ is in $G_m$). At each step in this process, at most $\lambda(G_0)^2\log (N\vert G_0\vert)$ of the $h_i$ are discarded. Additionally, $m$ is bounded by $2\log \vert G_0\vert$. As a result \begin{align*}N-n&<2\lambda(G_0)^2\log(N\vert G_0\vert)\log \vert G_0\vert<2\lambda(\ell)^2\log^2(\phi(\ell))\log N\\&<2N^{\frac{2}{5}}\log N<\frac{\sqrt{N}}{6}.\end{align*} In particular, $n>\frac{5N}{6}$. Let $q_{a_1}, q_{a_2}, \ldots, q_{a_{N-n}}$ be the primes which were discarded, and let $p_1, p_2, \ldots, p_n$ be the remaining primes, the largest of which we will call $p_0$. Let $b$ be the recentered version of $B$, that is, \[b=B+\frac{\sum_{i=1}^{N-n}\log q_{a_i}}{2}.\] Then \begin{align*}\vert b\vert&\leq \vert B\vert+\frac{\sum_{i=1}^{N-n}\log q_{a_i}}{2}\leq \vert B\vert+\frac{\sqrt{N}\log q_0}{12}\leq \frac{\sqrt{N}\log q_0}{9}\\&\leq \frac{2\sqrt{N}\log p_0}{9}<\frac{\sqrt{n}\log p_0}{4}.\end{align*}

In what follows, we will refer to $G_m$ as $G$, and we will write it additively, despite its origin as a subgroup of a multiplicative group. We now introduce some final notation before the main part of the proof begins. Let $K$ be the interval $[b-\frac{1}{2A},b+\frac{1}{2A}]$. Let $P=\prod_i p_i$. Let $r_i=\frac{\log p_i}{2}$ and $r_0=\frac{\log p_0}{2}$. Let $g_i$ be the representative of $p_i$ in $G$. Let $\chi_0$ be the trivial character of $G$. Let $X_i$ be the random variable on the space $G\times \mathbb{R}$ which is $(0,-r_i)$ or $(g_i,r_i)$ with equal probability.

Now, we wish to show that \[\text{Pr}_{d\mid P}\Big(\log d-\frac{\log P}{2}\in K,\text{ }\sum_{i: \text{ }p_i\mid d}g_i=0\Big)\] is positive, where $\text{Pr}_{d\mid P}$ denotes the probability of a random divisor of $P$ satisfying some condition. Let $\mathbbm{1}_e$ be the characteristic function of the identity in $G$ and $\mathbbm{1}_K$ be the characteristic function of the interval $K$ in $\mathbb{R}$. Note that for every $d$ that divides $P$, \[\sum_{i: \text{ }p_i\mid d}r_i+\sum_{i: \text{ }p_i\nmid d}(-r_i)=\sum_{i: \text{ }p_i\mid d}\log p_i-\frac{1}{2}\sum_i\log p_i=\log d-\frac{\log P}{2}.\] Therefore, we can write \[\text{Pr}_{d\mid P}\Big(\log d-\frac{\log P}{2}\in K,\text{ }\sum_{i: \text{ }p_i\mid d}g_i=0\Big)=\mathbb{E}\Big(\mathbbm{1}_{e\times K}\big(\sum_jX_j\big)\Big)\] where $\mathbbm{1}_{e\times K}=(\mathbbm{1}_e,\mathbbm{1}_K)$, in other words, $\mathbbm{1}_{e\times K}$ acts on $G$ as $\mathbbm{1}_e$ and acts on $\mathbb{R}$ as $\mathbbm{1}_K$. We are about to take a Fourier transform, so we will replace $\mathbbm{1}_K$ with a Gaussian, making the higher frequencies easier to deal with. Let $\Psi(x)=e^{-a^2(x-b)^2}$ where $a=2A\sqrt{\log \Phi}$ and let $\Psi_e=(\mathbbm{1}_e,\Psi)$. Now, \[\mathbbm{1}_e(g)=\frac{1}{\vert G\vert}\sum_{\chi\in \hat{G}}\chi(g)\] while \[\hat{\Psi}(\omega)=\int_{-\infty}^{\infty}\Psi(x)e^{-i\omega x}dx=\frac{\sqrt{\pi}}{a}e^{-\frac{\omega^2}{4a^2}-ib\omega}\] so \[\Psi(x)=\frac{1}{2\pi}\int_{-\infty}^{\infty}\hat{\Psi}(\omega)e^{i\omega x}d\omega=\frac{1}{2a\sqrt{\pi}}\int_{-\infty}^{\infty}e^{-\frac{\omega^2}{4a^2}-ib\omega}e^{i\omega x}d\omega.\] Thus \[\Psi_e(x,g)=\frac{1}{2a\sqrt{\pi}\vert G\vert}\sum_{\chi}\int_{-\infty}^{\infty}e^{-\frac{\omega^2}{4a^2}-ib\omega}\chi_{\omega}(x,g)d\omega\] where $\chi_{\omega}$ is the character on $G\times \mathbb{R}$ corresponding to $\chi$ on $G$ and $e^{i\omega x}$ on $\mathbb{R}$, that is, $\chi_{\omega}=\chi\boxtimes e^{i\omega x}$. Let \[C=2a\sqrt{\pi}\vert G\vert\] and \[E=\mathbb{E}\Big(\Psi_e\big(\sum_jX_j\big)\Big).\] Note that \[\mathbb{E}(\chi_{\omega}(X_j))=\frac{e^{i\omega r_j}\chi(g_j)+e^{-i\omega r_j}}{2}.\] We now claim that \begin{align*}CE&=C\mathbb{E}\Big(\frac{1}{2a\sqrt{\pi}\vert G\vert}\sum_{\chi}\int_{-\infty}^{\infty}e^{-\frac{\omega^2}{4a^2}-ib\omega}\chi_{\omega}\Big(\sum_j X_j\Big)d\omega\Big)\\&=\mathbb{E}\Big(\sum_{\chi}\int_{-\infty}^{\infty}e^{-\frac{\omega^2}{4a^2}-ib\omega}\prod_j\big(\chi_{\omega}(X_j)\big)d\omega\Big)\\&=\sum_{\chi}\int_{-\infty}^{\infty}e^{-\frac{\omega^2}{4a^2}-ib\omega}\prod_j\Big(\frac{e^{i\omega r_j}\chi(g_j)+e^{-i\omega r_j}}{2}\Big)d\omega\\&=\Re\bigg(\sum_{\chi}\int_{-\infty}^{\infty}e^{-\frac{\omega^2}{4a^2}-ib\omega}\prod_j\Big(\frac{e^{i\omega r_j}\chi(g_j)+e^{-i\omega r_j}}{2}\Big)d\omega\bigg).\end{align*} The second equality follows since the random variables $X_j$ are independent, while in the next step, we interchange the expectation operator with the integral and the product. (Note that the integrand has absolute value bounded by a Gaussian and is therefore absolutely convergent.) Finally, the last equality follows since $C$ and $E$ are both real.

We will show that the main contribution to the above expression comes from when $\chi=\chi_0$ and $\omega$ is very small. The challenge is to show that every other contribution is small. The main way we do this is by showing that $\prod_j\Big(\frac{e^{i\omega r_j}\chi(g_j)+e^{-i\omega r_j}}{2}\Big)$ is small. When $\chi$ is non-trivial and $\omega$ is very small, this product will be small if we can show that $\chi(g_j)$ is sometimes not equal to 1. Fortunately, $G$ was constructed entirely for that purpose. Then as $\omega$ becomes larger (in absolute value), the set $\{e^{i\omega r_j}\}_{j=1}^n$ should spread out, providing cancellation. Once $\omega$ reaches a certain size, the contribution becomes negligible for a different reason, namely because of the $e^{-\frac{\omega^2}{4a^2}}$ term.

Now, for values of $c_1$, $c_2$, $c_3$, and $c_4$ which will be chosen shortly, we can write that \[CE>2\Big(I_1+I_2-I_3-\sum_{\chi\neq \chi_0}J_1^{\chi}-\sum_{\chi}(J_2^{\chi}+J_3^{\chi})\Big)\] where \[I_1=\int_0^{c_1}e^{-\frac{\omega^2}{4a^2}}\cos(b\omega)\prod_j\cos(\omega r_j)d\omega,\] \[I_2=\int_{c_1}^{c_2}e^{-\frac{\omega^2}{4a^2}}\cos(b\omega)\prod_j\cos(\omega r_j)d\omega,\] and \[I_3=\int_{c_2}^{c_3}\prod_j\big\vert\cos(\omega r_j)\big\vert d\omega,\] while \[J_1^{\chi}=\max_{s\in \{\pm 1\}}\int_0^{c_3}\prod_j\Big\vert\frac{e^{is\omega r_j}\chi(g_j)+e^{-is\omega r_j}}{2}\Big\vert d\omega,\] \[J_2^{\chi}=\max_{s\in \{\pm 1\}}\int_{c_3}^{c_4}\prod_j\Big\vert\frac{e^{is\omega r_j}\chi(g_j)+e^{-is\omega r_j}}{2}\Big\vert d\omega,\] and \[J_3^{\chi}=\int_{c_4}^{\infty}e^{-\frac{\omega^2}{4a^2}}d\omega.\] (Note that the maximum over $s$ in $J_1^{\chi}$ and $J_2^{\chi}$ comes from having to account for the contribution from negative values of $\omega$.)

Again, the claim is that the main contribution comes from $I_1$. We take $c_1=\frac{1}{r_0\sqrt{n}}$, which will be small enough to ensure that the minimum value of the integrand over that range will still be large. Then $c_2$ will be chosen to be $\frac{3}{r_0\sqrt{n}}$, ensuring that $I_2$ is at least positive. Next, $c_3=\frac{1}{\lambda(G)r_0}$. (Since $n>9\lambda(G)^2$, $c_2$ is indeed less than $c_3$.) Elementary methods will show that $I_3$ is small compared to $I_1$. Then $J_1^{\chi}$ can be shown to be quite small because of the work done previously in making sure that there is no proper subgroup of $G$ which contains almost all of the $g_i$. Bounding $J_3^{\chi}$ will be quite easy as long as $c_4$ is large enough for the Gaussian term to become vanishingly small. To that end, we take $c_4=4A\log \Phi$. That leaves the task of bounding $J_2^{\chi}$. A priori, this part of the integral is very difficult to control because it depends on there being no grand conspiracy in the choice of our primes. This is where our non-clustering hypothesis \eqref{nocluster} comes in.

We will start with a lower bound on $I_1$. Let $0\leq \omega\leq c_1$. By assumption, $A>1$ so $a>1$. Therefore, $e^{-\frac{\omega^2}{4a^2}}>\frac{99}{100}$ for large $n$. Also, \[\cos(b\omega)\geq \cos(\frac{1}{2})>\frac{5}{6}\] since \[\vert b\omega\vert\leq \frac{\sqrt{n}\log p_0}{4r_0\sqrt{n}}=\frac{1}{2}.\] Now \[\cos(\omega r_j)\geq 1-\frac{\omega^2 r_j^2}{2}\geq 1-\frac{1}{2n}\] so \[\prod_{j=1}^n(\cos(\omega r_j))\geq \Big(1-\frac{1}{2n}\Big)^n>\frac{20}{33}\] for large $n$. Hence \[I_1=\int_0^{c_1}e^{-\frac{\omega^2}{4a^2}}\cos(b\omega)\prod_j\big(\cos(\omega r_j)\big)d\omega>\int_0^{c_1}\frac{1}{2}d\omega=\frac{1}{2r_0\sqrt{n}}.\]

Now, the integrand \[e^{-\frac{\omega^2}{4a^2}}\cos(b\omega)\prod_j(\cos(\omega r_j))\] will be positive if $\vert b\omega\vert<\frac{\pi}{2}$ (which will be true if $\vert \omega\vert<\frac{\pi}{r_0\sqrt{n}}$) and $\vert \omega r_j\vert<\frac{\pi}{2}$ for all $j$ (which will be true if $\vert \omega\vert<\frac{\pi}{2r_0}$). Since $c_2=\frac{3}{r_0\sqrt{n}}$, it follows that $I_2>0$.

Now suppose $c_2\leq \vert\omega\vert\leq c_3$. Then $\vert\omega r_j\vert\leq 1$ for all $j$. Hence \[I_3=\int_{c_2}^{c_3}\prod_j\big\vert\cos(\omega r_j)\big\vert d\omega<\int_{c_2}^{c_3}\Big(1-\frac{\omega^2r_0^2}{4}\Big)^nd\omega<\int_{c_2}^{c_3}e^{-\frac{n\omega^2r_0^2}{4}}d\omega.\] Letting $u=\frac{\sqrt{n}\omega r_0}{2}$, \[I_3<\frac{2}{r_0\sqrt{n}}\int_{\frac{3}{2}}^{\frac{\sqrt{n}}{2\lambda(G)}}e^{-u^2}du<\frac{2}{r_0\sqrt{n}}\int_{\frac{3}{2}}^{\infty}e^{-u^2}du=\frac{\sqrt{\pi}}{r_0\sqrt{n}}\text{erfc}\big(\frac{3}{2}\big)<\frac{1}{5r_0\sqrt{n}},\] where \[\text{erfc}(x)=\frac{2}{\sqrt{\pi}}\int_x^{\infty}e^{-t^2}dt\] with the elementary bound $\text{erfc}(x)\leq e^{-x^2}$ for $x\geq 0$.

Let $\chi$ be any non-trivial character. Take $s$ to be either $-1$ or $1$. Let $0\leq \vert\omega\vert\leq c_3$. Note that \[\Big\vert \frac{\chi(g_j)+e^{-2is\omega r_j}}{2}\Big\vert=\sqrt{\frac{1+\Re(\chi(g_j)e^{2is\omega r_j})}{2}}=\sqrt{\frac{1+\cos(\frac{2\pi k_j}{\lambda(G)}+2s\omega r_j)}{2}}\] where $-\frac{\lambda(G)}{2}<k_j\leq \frac{\lambda(G)}{2}$ is the integer such that $\chi(g_j)=e^{\frac{2\pi ik_j}{\lambda(G)}}$. Indeed, such an integer exists since the order of $g_j$ divides $\lambda(G)$, so $\chi(g_j)^{\lambda(G)}=1$. If $\vert k_j\vert\geq 1$ then \[\Big\vert \frac{2\pi k_j}{\lambda(G)}+2s\omega r_j\Big\vert\geq \Big\vert \frac{2\pi k_j}{\lambda(G)}\Big\vert-\Big\vert 2s\omega r_j\Big\vert\geq \frac{2\pi}{\lambda(G)}-\frac{2\pi}{3\lambda(G)}=\frac{4\pi}{3\lambda(G)}>\frac{4}{\lambda(G)}\] because $\vert \omega r_j\vert\leq \frac{1}{\lambda(G)}<\frac{\pi}{3\lambda(G)}$. On the other hand, \[\Big\vert \frac{2\pi k_j}{\lambda(G)}+2s\omega r_j\Big\vert\leq \pi+\frac{2}{\lambda(G)}.\] Therefore, if $\vert k_j\vert\geq 1$ then the distance between $\frac{2\pi k_j}{\lambda(G)}+2s\omega r_j$ and the nearest multiple of $2\pi$ is greater than $\frac{4}{\lambda(G)}$ because then $\lambda(G)\geq 2$ and even in the extreme case that $\lambda(G)=2$, \[2\pi-\big(\pi+\frac{2}{\lambda(G)}\big)=\pi-1>2=\frac{4}{\lambda(G)}.\] Hence, \begin{align*}\Big\vert \frac{\chi(g_j)+e^{-2is\omega r_j}}{2}\Big\vert&<\sqrt{\frac{1+\cos\big(\frac{4}{\lambda(G)}\big)}{2}}<\sqrt{\frac{1+1-\frac{4}{\lambda(G)^2}}{2}}\\&=\sqrt{1-\frac{2}{\lambda(G)^2}}<1-\frac{1}{\lambda(G)^2}\end{align*} if $\vert k_j\vert\geq 1$. Recall that $G$ was constructed so that for every proper subgroup of $G$, at least $\lambda(G)^2\log(n\vert G\vert)$ of the $g_j$ are outside it. The kernel of a non-trivial character is a proper subgroup. Hence, for at least $\lambda(G)^2\log(n\vert G\vert)$ values of $j$, $\chi(g_j)\neq 1$, or alternatively, $\vert k\vert\geq 1$. Therefore, since even when $k_j=0$, the corresponding term in the product is still at most 1, \begin{align*}J_1^{\chi}&=\max_{s\in \{-1,1\}}\int_0^{c_3}\prod_j\Big\vert\frac{e^{is\omega r_j}\chi(g_j)+e^{-is\omega r_j}}{2}\Big\vert d\omega\\&<\int_0^{c_3}\Big(1-\frac{1}{\lambda(G)^2}\Big)^{\lambda(G)^2\log(n\vert G\vert)}d\omega<\int_0^{c_3}e^{-\log(n\vert G\vert)}d\omega\\&=\frac{e^{-\log(n\vert G\vert)}}{\lambda(G)r_0}=\frac{1}{n\vert G\vert\lambda(G)r_0}<\frac{1}{100\vert G\vert r_0\sqrt{n}}.\end{align*}

Now let $\chi$ be any character. Seeing as $\lambda(G)\leq \lambda(\ell)$, the non-clustering hypothesis \eqref{nocluster} implies that for every \[\frac{2}{\lambda(G)\log p_0}\leq \vert \omega\vert\leq 4A\log\Phi,\] at least a third of the possible values of $j$ are such that \[\Big\vert 2\omega r_j-\frac{2\pi k}{\lambda(G)}\Big\vert\geq 8\sqrt{\frac{\log \Phi}{N}}>7\sqrt{\frac{6}{5N}\log \Phi}>7\sqrt{\frac{\log \Phi}{n}}\] for all integers $k$. (This would have been true for at least half of the original primes, but we have since discarded up to a sixth of the primes.) For $\omega$ in the same interval, it follows that letting $k$ be the integer that minimizes $2s\omega r_j-\frac{2\pi k}{\lambda(G)}$ for $s=1$, \[\cos\Big(2s\omega r_j-\frac{2\pi k}{\lambda(G)}\Big)<\cos\Big(7\sqrt{\frac{\log \Phi}{n}}\Big)<1-\frac{49\log \Phi}{4n}<1-\frac{12\log \Phi}{n}.\] Moreover, the same inequality holds for $s=-1$ and $k$ which minimizes $-2\omega r_j-\frac{2\pi k}{\lambda(G)}$ since cosine is even. We have thus established an upper bound on $\cos\big(2s\omega r_j-\frac{2\pi k}{\lambda(G)}\big)$ for all integers $k$ and both choices of $s$. Hence, it follows that in this range of $\omega$ between $c_3$ and $c_4$, \[\Big\vert \frac{\chi(g_j)+e^{-2is\omega r_j}}{2}\Big\vert=\sqrt{\frac{1+\cos(\frac{2\pi k_j}{\lambda(G)}+2s\omega r_j)}{2}}<\sqrt{1-\frac{6\log \Phi}{n}}<1-\frac{3\log \Phi}{n}\] where $k_j$ is chosen to be an integer such that $\chi(g_j)=e^{\frac{2\pi ik_j}{\lambda(G)}}$. Therefore \begin{align*}J_2^{\chi}&=\max_{s\in \{-1,1\}}\int_{c_3}^{c_4}\prod_j\Big\vert\frac{e^{is\omega r_j}\chi(g_j)+e^{-is\omega r_j}}{2}\Big\vert d\omega<\int_{c_3}^{c_4}\Big(1-\frac{3\log \Phi}{n}\Big)^{\frac{n}{3}}d\omega\\&<\int_{c_3}^{c_4}e^{-\log \Phi}d\omega<c_4e^{-\log \Phi}=\frac{4A\log \Phi}{\Phi}=\frac{4A}{\sqrt{\Phi}}\frac{\log\Phi}{\sqrt{\Phi}}\\&<\frac{4}{N\phi(\ell)\log q_0}<\frac{1}{100\vert G\vert r_0\sqrt{n}},\end{align*} recalling that $\sqrt{\Phi}=AN\phi(\ell)\log q_0$.

Finally, letting $u=\frac{\omega}{2a}$, \begin{align*}J_3^{\chi}&=\int_{c_4}^{\infty}e^{-\frac{\omega^2}{4a^2}}d\omega=2a\int_{\frac{c_4}{2a}}^{\infty}e^{-u^2}du=a\sqrt{\pi}\text{ erfc}\Big(\frac{c_4}{2a}\Big)\leq a\sqrt{\pi}e^{-\frac{c_4^2}{4a^2}}\\&=a\sqrt{\pi}e^{-\frac{16A^2\log^2\Phi}{16A^2\log\Phi}}=\frac{a\sqrt{\pi}}{\Phi}=\frac{2\sqrt{\pi}A}{\sqrt{\Phi}}\sqrt{\frac{\log\Phi}{\Phi}}\\&<\frac{2\sqrt{\pi}}{N\phi(\ell)\log q_0}<\frac{1}{100\vert G\vert r_0\sqrt{n}}.\end{align*}

Putting everything together, \[CE>\frac{2}{r_0\sqrt{n}}\Big(\frac{1}{2}-\frac{1}{5}-\frac{\vert G\vert}{100\vert G\vert}-\frac{\vert G\vert}{100\vert G\vert}-\frac{\vert G\vert}{100\vert G\vert}\Big)>\frac{1}{2r_0\sqrt{n}}\] so \[E>\frac{1}{2r_0\sqrt{n}}\frac{1}{2a\sqrt{\pi}\vert G\vert}>\frac{1}{15A\vert G\vert r_0\sqrt{n\log \Phi}},\] using that $a=2A\sqrt{\log \Phi}$. Now, \begin{align*}\text{Pr}_{d\mid P}\Big(\log d-\frac{\log P}{2}\in K,\text{ }\sum_{i: \text{ }p_i\mid d}g_i=0\Big)&=\mathbb{E}\Big(\mathbbm{1}_{e\times K}\big(\sum_jX_j\big)\Big)\\&=\mathbb{E}\Big(\Psi_e\big(\sum_jX_j\big)\Big)-\mathbb{E}\Big(\Psi_e\big(\sum_jX_j\big)-\mathbbm{1}_{e\times K}\big(\sum_jX_j\big)\Big)\\&\geq E-\max(\Psi_e-\mathbbm{1}_{e\times K})=E-e^{-\frac{a^2}{4A^2}}=E-e^{-\log\Phi}\\&>\frac{1}{15A\vert G\vert r_0\sqrt{n\log\Phi}}-\frac{1}{AN\phi(\ell)\log q_0\sqrt{\Phi}}\\&>\frac{1}{15A\vert G\vert r_0\sqrt{n\log\Phi}}-\frac{1}{100A\vert G\vert r_0\sqrt{n\log\Phi}}\\&>\frac{1}{18A\vert G\vert r_0\sqrt{n\log \Phi}}.\end{align*} Hence there are at least \[\frac{2^n}{18A\vert G\vert r_0\sqrt{n\log \Phi}}>\frac{2^n}{\Phi}=\frac{2^{N-O(\sqrt{N})}}{\Phi}\] values of $d$ dividing $P$ (and hence $Q$) such that $d\equiv 1\pmod{\ell}$ and \[\pushQED{\qed}\big\vert \log d-\frac{\log Q}{2}-B\big\vert=\big\vert \log d-\frac{\log P}{2}-b\big\vert\leq\frac{1}{2A}.\qedhere\popQED\]
\hphantom\qedhere
\end{proof}
The proof of this technical result out of the way, we now return to the main argument. In Section 3, we constructed a set $\bold{Q}$, consisting of pairs of ``non-clustering'' primes. These primes are going to be the prime divisors of our Carmichael numbers. 

Let $N_-=\lceil \frac{1}{4}e^{\frac{\eta y}{2\log y}}\rceil $, and let $N_+=\lfloor e^{\frac{\eta y}{2\log y}}\rfloor$. Recall that $\vert\bold{Q}\vert\geq N_+$. We label the first $N_+$ pairs in $\bold{Q}$ as $(q_1,q'_1),\ldots,(q_{N_+},q_{N_+}')$. For $N\leq N_+$, let $\bold{Q}_N$ be the set of primes of the form $q_i$ or $q'_i$ with $1\leq i\leq N$. Remember that these primes are all of the form $dk_0+1$ for some $d\mid L$ and hence are distinct from each other. Let $Q_N=\prod_{q\in \bold{Q}_N}q$. Let \[Z_-(y)=\Big\lceil e^{y^{2+2\Delta}}\Big\rceil^{\frac{1}{3}e^{\frac{\eta y}{2\log y}}}\] and \[Z_+(y)=\Big\lceil e^{y^{2+2\Delta}}\Big\rceil^{\frac{11}{12}e^{\frac{\eta y}{2\log y}}}.\] Let $z$ be an arbitrary integer between $Z_-(y)$ and $Z_+(y)$. First we claim that there exists an integer $n$ between $N_-$ and $N_+$ such that \[\big\vert \log z-\frac{\log Q_n}{2}\big\vert<\log Y.\] Indeed, recalling that $Y=\Upsilon L$ where $\Upsilon=\big\lceil e^{y^{2+2\Delta}}\big\rceil$ and \[L=\prod_{p\in\bold{P}}p<y^{\frac{2y}{\log y}}=e^{2y},\] it follows that \[Q_{N_+}>Y^{2e^{\frac{\eta y}{2\log y}}}>(Z_+(y))^2\] and \[Q_{N_-}<(2YL)^{\frac{1}{2}e^{\frac{\eta y}{2\log y}}}<\Upsilon^{\frac{2}{3}e^{\frac{\eta y}{2\log y}}}=(Z_-(y))^2,\] while \[\frac{\log Q_{N+1}-\log Q_N}{2}<\frac{\log((2LY)^2)}{2}=\log(2LY)<2\log Y\] for all $N<N_+$. Therefore, such $n$ does indeed exist. Let \[b=\log z-\frac{\log Q_n}{2}+\frac{1}{2e^{\frac{\eta y}{(4+6\Delta)\log y}}}.\] Then $\vert b\vert<\log Y+1$.
\begin{lemma}
\label{technical}
The hypotheses of Theorem \ref{Fourier} hold for $N=2n$, $B=b$, $\ell=L$, $A=e^{\frac{\eta y}{(4+6\Delta)\log y}}$, and $\{q_i\}=\bold{Q}_n$.
\end{lemma}
\begin{proof}
Since $L$ is small compared to $Y$, it is indeed true that for all $i$, $q_i^2\geq q_0$, where $q_0$ is the largest element of $\bold{Q}_n$. Also, $N>\log^5 q_0$. Letting $v_q(m)$ be the $p$-adic valuation of $m$, \begin{align}\label{lambda}\lambda(L)&\notag=\prod_{\substack{q\leq y^{1-E}\\ q\in\mathbb{P}}}\max_{p\in \bold{P}}q^{v_q(p-1)}\leq \prod_{\substack{q\leq y^{1-E}\\ q\in\mathbb{P}}}y\\&=y^{\pi(y^{1-E})}<y^{\frac{2y^{1-E}}{(1-E)\log y}}=e^{\frac{2y^{1-E}}{1-E}}.\end{align} Combining this with the fact that \[\log \phi(L)<\log L<\log \big(y^{\frac{2y}{\log y}}\big)=2y,\] it follows that \[N\geq \Big(\lambda(L)\log \phi(L)\Big)^5.\] Also, $\vert B\vert<\log Y+1<\frac{\sqrt{N}\log q_0}{36}$.

Now, considering how $\bold{Q}$ was constructed, at most half of the primes $q\in \bold{Q}_n$ are such that \[\min_{\substack{k\in \mathbb{Z}\\ k\neq 0}}\big(\vert \omega\log q-k\vert\big)\leq\frac{1}{e^{\frac{\eta y}{(4+2\Delta)\log y}}}\] for any particular $\omega$ with $\vert \omega\vert\leq e^{\frac{\eta y}{(4+4\Delta)\log y}}$. As a consequence of this and \eqref{lambda}, for all $\omega$ with $\vert \omega\vert\leq e^{\frac{\eta y}{(4+5\Delta)\log y}}$ and $\alpha\geq \frac{2\pi}{\lambda(L)}$, at most half of the primes $q\in \bold{Q}_n$ are such that \[\vert \omega\log q-k\alpha\vert\leq \frac{1}{e^{\frac{\eta y}{(4+\Delta)\log y}}}\] for some non-zero integer $k$. Additionally, if $\vert \omega\vert\geq \frac{2}{\lambda(L)\log q_0}$ then \[\vert \omega\log q\vert\geq \frac{1}{4\lambda(L)}>\frac{1}{4e^{\frac{2y^{1-E}}{1-E}}}>\frac{1}{e^{\frac{\eta y}{(4+\Delta)\log y}}}.\] Therefore, for $\frac{2}{\lambda(L)\log q_0}\leq \vert\omega\vert\leq e^{\frac{\eta y}{(4+5\Delta)\log y}}$ and $\alpha\geq \frac{2\pi}{\lambda(L)}$, at most half of the primes $q\in \bold{Q}_n$ are such that \[\vert \omega\log q-k\alpha\vert\leq \frac{1}{e^{\frac{\eta y}{(4+\Delta)\log y}}}\] for any integer $k$. Let $\Phi=(AN\phi(L)\log q_0)^2$. Then \[\log\Phi=2\big(\log A+\log N+\log \phi(L)+\log\log q_0\big)<4\Big(\frac{\eta y}{\log y}+y+\log\log Y\Big)<8y.\] Hence \[8\sqrt{\frac{\log\Phi}{N}}<\frac{32\sqrt{y}}{e^{\frac{\eta y}{4\log y}}}<\frac{1}{e^{\frac{\eta y}{(4+\Delta)\log y}}}\] and \[4A\log \Phi<32Ay<e^{\frac{\eta y}{(4+5\Delta)\log y}}.\] Therefore, for all $\omega$ with $\frac{2}{\lambda(L)\log q_0}\leq \vert \omega\vert\leq 4A\log\Phi$ and $\alpha\geq \frac{2\pi}{\lambda(L)}$, at most half of the primes $q\in \bold{Q}_n$ are such that \[\vert \omega\log q-k\alpha\vert<8\sqrt{\frac{\log \Phi}{N}}\] for some integer $k$. 
\end{proof}
Let $\Pi$ be one of the products guaranteed by the theorem. Each of its prime factors is congruent to 1 mod $k_0$. Hence $\Pi\equiv 1\pmod{k_0}$. By construction, it is also true that $\Pi\equiv 1\pmod{L}$. Since $(k_0,L)=1$, these can be combined: $\Pi\equiv 1\pmod{k_0L}$. Each prime factor of $\Pi$ looks like $dk_0+1$ for some $d\mid L$. It follows that $p-1\mid \Pi-1$ for every $p\mid \Pi$, which means that $\Pi$ is a Carmichael number. Also, \[\Big\vert \log \Pi-\log z-\frac{1}{2A}\Big\vert=\Big\vert \log \Pi-\frac{\log Q_n}{2}-B\Big\vert\leq\frac{1}{2A}.\] Therefore, there exist at least \[2^{2N_--O(\sqrt{N_+}+\log\Phi)}>e^{\frac{1}{3}e^{\frac{\eta y}{2\log y}}}\] Carmichael numbers with a logarithm in the interval $(\log z,\log z+\frac{1}{A})$. We now want to remove all mention of $y$ in the quantities in the preceding statement and write everything in terms of $z$. We have that \[z\leq \Big\lceil e^{y^{2+2\Delta}}\Big\rceil^{\frac{11}{12}e^{\frac{\eta y}{2\log y}}}\leq e^{e^{\frac{\eta y}{(2-\frac{\delta}{2})\log y}}}\] so recalling that $\Delta=\frac{\delta}{24}$, \[\frac{A}{2}=\frac{e^{\frac{\eta y}{(4+6\Delta)\log y}}}{2}>\frac{1}{2}(\log z)^{\frac{2-\frac{\delta}{2}}{4+\frac{\delta}{4}}}>(\log z)^{\frac{1}{2+\delta}}\] since our assumption that $\delta$ is small means that \[(2-\frac{\delta}{2})(2+\delta)=4+\delta-\frac{\delta^2}{2}>4+\frac{\delta}{4}.\] Therefore \[e^{\log z+\frac{1}{A}}<z\Big(1+\frac{2}{A}\Big)<z+\frac{z}{(\log z)^{\frac{1}{2+\delta}}}.\] Also, $\log z<y^{2+2\Delta}e^{\frac{\eta y}{2\log y}}$ and $\log\log z>y^{1-\Delta}$ so \[\frac{\log z}{(\log\log z)^{2+\delta}}<\frac{y^{2+2\Delta}e^{\frac{\eta y}{2\log y}}}{y^{(1-\Delta)(2+\delta)}}=e^{\frac{\eta y}{2\log y}}y^{4\Delta+\Delta\delta-\delta}=e^{\frac{\eta y}{2\log y}}y^{\frac{\delta^2}{24}-\frac{5\delta}{6}}<\frac{1}{3}e^{\frac{\eta y}{2\log y}},\] which means that $e^{\frac{1}{3}e^{\frac{\eta y}{2\log y}}}>e^{\frac{\log z}{(\log\log z)^{2+\delta}}}$. Therefore, there are at least $e^{\frac{\log z}{(\log\log z)^{2+\delta}}}$ Carmichael numbers in the interval $(z,z+\frac{z}{(\log z)^{\frac{1}{2+\delta}}})$ for all $z$ between $Z_-(y)$ and $Z_+(y)$. Note that \[Z_-(y+1)=\Big\lceil e^{(y+1)^{2+2\Delta}}\Big\rceil^{\frac{1}{3}e^{\frac{\eta(y+1)}{2\log(y+1)}}}<e^{\frac{101}{100}\frac{e}{3}y^{2+2\Delta}e^{\frac{\eta y}{2\log y}}}<e^{\frac{11}{12}y^{2+2\Delta}e^{\frac{\eta y}{2\log y}}}\leq Z_+(y).\] As a result, the intervals $[Z_-(y),Z_+(y)]$ and $[Z_-(y+1),Z_+(y+1)]$ overlap. Hence, sliding $y$ to infinity by increments of one, we conclude that for sufficiently large $z$, there are at least $e^{\frac{\log z}{(\log\log z)^{2+\delta}}}$ Carmichael numbers in the interval $(z,z+\frac{z}{(\log z)^{\frac{1}{2+\delta}}})$. As a final remark, we note that Theorem \ref{BV} is effective and therefore so is our main result.
\section{Acknowledgements}
The author would like to acknowledge a helpful letter from James Maynard and a number of useful conversations with Michael Larsen. The author would also like to thank Andrew Granville, Simon H\"{o}ra, and the anonymous referee for their many helpful comments and suggestions. 

\end{document}